\documentclass[a4paper,11pt]{article}
\usepackage{epsf,graphicx,epsfig}
\usepackage{amscd}
\usepackage{amsmath,latexsym,amssymb,amsthm}
\usepackage[nospace,noadjust]{cite}
\usepackage{textcomp}
\usepackage{setspace,cite}
\usepackage{lscape,fancyhdr,fancybox}
\usepackage{stmaryrd}
\usepackage[all,cmtip]{xy}
\usepackage{tikz}
\usepackage{cancel}
\usetikzlibrary{shapes,arrows,decorations.markings}
\usepackage{graphicx}

\usepackage{color}
\usepackage{url}
\usepackage{enumerate}
\usepackage[mathscr]{euscript}

  \theoremstyle{plain}

    \newtheorem{theorem}{Theorem}[section]
    \newtheorem{prop}[theorem]{Proposition}

\theoremstyle{definition}
    \newtheorem{definition}[theorem]{Definition}
        \newtheorem{remark}[theorem]{Remark}
\theoremstyle{remark}

\newcommand{\Hom}{\operatorname{Hom}}

\newcommand{\fg}{\mathfrak g}

\begin{document}

\allowdisplaybreaks

\title{Representation and cohomology of embedding tensors on Malcev algebras}

\author{Tao Zhang, \quad Wei Zhong}

\date{}

\maketitle

\footnotetext{{\it{Keyword}: Embedding tensor, Malcev algebra, representation, cohomology, deformations.}}

\footnotetext{{\it{Mathematics Subject Classification (2020)}}: 17D10, 17B38, 17B56.}

\begin{abstract}
We introduce the concept of embedding tensor on Malcev algebras.
The  representation and cohomology theory of embedding tensor on Malcev algebras are studied.
Some applications in deformation and abelian extension are investigated.
\end{abstract}

\section{Introduction}
Embedding tensors arise in the study of supergravity theory and higher gauge theories.
The mathematical concept "embedding tensor" is associated with tensor hierarchies which plays a very important role in the construction of supergravity theories and  higher gauge theories.
For reference about embedding tensors, Lie-Leibniz triples and associated tensor hierarchies, see \cite{bon2,wit,wit2,wit3,wit4,kotov-strobl,nicolai,Pal}.
Given a Lie algebra $(\mathfrak{g}, [~,~])$ and a representation $(V, \rho)$, a linear map $T: V \rightarrow \mathfrak{g}$ is said to be an embedding tensor on $\mathfrak{g}$ with respect to the representation $(V, \rho)$ if the map $T$ satisfies the following identity:
\begin{align}
    [T(u), T(v)] = T (\rho (Tu) v), \text{ for } u, v \in V.
\end{align}
Recently, the mathematical theory of embedding tensors on Lie algebras, Hom-Lie algebras and 3-Lie algebras was studied in \cite{sheng-embd,TY,Das1,HHSZ}.
It is showed that an embedding tensor gives rise to a Leibniz algebra that further gives rise to higher algebra structure,  see \cite{kotov-strobl,lavau,lavau-p,lavau-stas}.

On the other hand, Malcev algebras were introduced by Malcev \cite{Mal}, who called these objects Moufang-Lie algebras.
Malcev algebras play a key role in the geometry of smooth loops since the tangent algebra of a locally analytic Moufang loop is a Malcev algebra.
For general theory about Malcev algebras, see \cite{E1,Fi,Sag}.
In \cite{Yam1},  it is proved that any Lie algebra $G$ with a direct sum decomposition $G=D\oplus M$ with $D$ is subalgebra of $G$ and
$[D, M]\subset M$,  then there is a Malcev algebra on $M$.
In \cite{Yam2},  Yamaguti introduced a cohomology theoy for a Malcev algebra $\fg$ with coefficients in a representation $M$.
Recently, Kupershmidt operators on Malcev algebras were studied  in \cite{Mab1,Mab}.

Since Malcev algebras are generalization of Lie algebras, it is natural to extend the theory of embedding tensor on Lie algebras to the Malcev algebras setting.
In the present paper, we introduce the concept of embedding tensor on Malcev algebras.
We will show that this new concept of embedding tensor on Malcev algebras give rise to Malcev dialgebras in \cite{Bre}.
We also investigate the representation and cohomology  theory for this type of Malcev algebras.
As applications, we will study the infinitesimal deformation and abelian extension theory of embedding tensor on Malcev algebras.

The organization of this paper is as follows.
In Section 2, we review some basic facts about Malcev algebras and Malcev dialgebras.
We introduced the concept of embedding tensor Malcev algebras and study its elementary properties.
In Section 3, we investigate the representation theory of embedding tensors on Malcev algebras.
We show that given an embedding tensor and a representation,we can construct a new embedding tensor on their direct sum spaces which is called  the semidirect product.
In  Section 4, we study abelian extensions of embedding tensors on Malcev algebras.
It is proved that equivalent classes of abelian extensions are one-to-one correspondence to the elements of the second cohomology groups.
In  the last Section 5, we study infinitesimal deformations of embedding tensor Malcev algebras.
The notion of Nijenhuis operators is introduced to describe trivial deformations.

Throughout the rest of this paper, we work over a fixed field of characteristic 0.
The space of linear maps from a vector space  $V$ to $W$ is denoted by $\Hom(V,W)$.
\section{Embeding tensor on Malcev algebras}
In this section, we introduce the concept of embedding tensor on Malcev algebras  and give its elementary properties.

\begin{definition}[\cite{Mal}]
A Malcev algebra is a vector space $\fg$ endowed with a skew-symmetric bracket $[~,~]: M\times M\to M$ satisfying the Malcev identity
\begin{align}
J(x,y,[x,z])=[J(x,y,z),x],
\end{align}
for all $x, y, z \in \fg$, where $J(x,y,z)=\big[[x,y],z\big]+\big[[y,z],x\big]+\big[[z,x],y\big]$ is the Jacobiator of $x,y,z.$
\end{definition}

\begin{remark}[\cite{Sag}]
 The Malcev identity is equivalent to Sagle's identity:
\begin{align}
[[x, z],[y, t]]=[[[x, y], z], t]+[[[y, z], t], x]+[[[z, t], x], y]+[[[t, x], y], z],
\end{align}
for all $x, y, z, t \in \fg$.
\end{remark}

Let $(\fg,[~,~])$ and $\left(\fg^{\prime},[~,~]^{\prime}\right)$ be two Malcev algebras. A linear map $\phi: \fg \rightarrow \fg^{\prime}$ is said to be a homomorphism of Malcev algebras if
$[\phi(x), \phi(y)]^{\prime}=\phi([x, y]), \text { for all } x, y \in \fg.$

\begin{definition}
 A representation of a Malcev algebra $\fg$ on a vector space $M$ is a map $\rho: \fg \to \operatorname{End}(M)$ such that
\begin{align}\label{def:rep}
\rho([[x, y], z])=\rho(x) \rho(y) \rho(z)-\rho(z) \rho(x) \rho(y)+\rho(y) \rho([z, x])-\rho([y, z]) \rho(x),
\end{align}
for all $x, y, z \in \fg$.
\end{definition}

Given a representation of a Malcev algebra $\fg$ on $M$. Define on the direct sum $\fg \oplus M$ the bracket by
\begin{align}
{\left[x+m, y+n\right]} & \triangleq\left[x, y\right]+\rho(x)(m)-\rho(y)(n).
\end{align}
Then  $\fg\oplus M$ is a Malcev algebra.
This is called the semidirect product of a Malcev algebra $\fg$ and $M$.

\begin{definition}
 Let $(\fg,[~,~ ])$ be a  Malcev algebra and $(M , \rho)$ be a representation of $\fg$. A linear map $T : M \rightarrow \fg$ is called an embedding tensor operator if it satisfying the following condition
 \begin{align}\label{def:et}
 [T(m),T(n)]=T \big(\rho(T(m))(n)\big), \quad  \forall m, n \in M.
 \end{align}
 In this case, $(\fg,M, T)$ is called an embedding tensor on Malcev algebra $\fg$.
\end{definition}

A right Malcev dialgebra was introduced in \cite{Bre}.
In this paper, we introduce the left one.
\begin{definition}
A left Malcev dialgebra is a vector space with a non-skew-symmetric bracket $[~,~]: M\times M\to M$ satisfying right anticommutativity and the
following identity:
\begin{align}
&[[x,y]+[y,x],z]=0,\\
&[x,[y,[z,t]]]=[y,[z,[x,t]]]+[z,[[x,y],t]]+[[x,z],[y,t]] + [[x,[y,z]],t].
\end{align}
\end{definition}

A simple example of  Malcev dialgebra is given as follows.
\begin{prop}\label{prop-hemi}
Let  $\fg$ be a Malcev algebra and $(M,\rho)$ is a representation of $\fg$. Define on the direct sum $\fg \oplus M$ the bracket by
\begin{align}\label{eq001}
{\left[x+m, y+n\right]_{H} } & \triangleq\left[x, y\right]+\rho(x)(n).
\end{align}
Then  $\fg\oplus M$ is a left Malcev dialgebra under the above bracket \eqref{eq001}.
This is called the hemi-semidirect product of a Malcev algebra $\fg$ and $M$ which is denoted by $\mathfrak{g} \ltimes_H M$.
\end{prop}
\begin{proof}
We are going to show that $\mathfrak{g} \ltimes_H M$ is a left Malcev dialgebra. In fact
\begin{align*}
&[ [x+m, y+n]_{H}+[y+n, x+m ]_{H}, z+p]_{H}\\
=&[ [x, y ]+\rho(x)(m)+ [y, x ]+\rho(y)(n),z+p]_{H}\\
=&[0+\rho(x)(m)+\rho(y)(n),z+p]_{H}\\
=&0.
\end{align*}
By definition, we get
\begin{align*}
&[x+m,[y+n,[z+p,t+q]_{H}]_{H}]_{H}\\
=&[x+m,[y+n,[z,t]+\rho(z)(q)]_{H}]_{H}\\
=&[x,[y,[z,t]]]+\rho(x)\rho(y)\rho(z)(q)
\end{align*}
and
\begin{align*}
&[y+n,[z+p,[x+m,t+q]_{H}]_{H}]_{H}+[z+p,[[x+m,y+n]_{H},t+q]_{H}]_{H}\\
&+[[x+m,z+p]_{H},[y+n,t+q]_{H}]_{H} + [[x+m,[y+n,z+p]_{H}]_{H},t+q]_{H}\\
=&[[y,[z,[x,t]]+[z,[[x,y],t]]+[[x,z],[y,t]] + [[x,[y,z]],t]\\
&+\rho(y)\rho(z)\rho(x)(q)+\rho(z)\rho([x,y])(q)\\
&+\rho([x,z])\rho(y)(q)+\rho([x,[y,z]])(q).
\end{align*}
Thus the right hand side terms of the above two equations are cancelled since $M$ is a representation of $\fg$.
Therefore we obtain that $\mathfrak{g} \ltimes_H M$ is a left Malcev dialgebra.
This complete the proof.
\end{proof}

\begin{prop}\label{o-char}
A linear map $T: M \rightarrow \mathfrak{g}$ is an $\mathcal{O}$-operator on $M$ over $\mathfrak{g}$ if and only if the graph
\begin{align*}
\mathrm{Gr} (T) := \{ (T(m), m) | m \in M \} \subset \mathfrak{g} \oplus M
\end{align*}
is a subalgebra of the hemi-semidirect product $\mathfrak{g} \ltimes_H M$.
\end{prop}
\begin{proof}
For any $m,n\in M$, we have
\begin{align*}
[(T(m), m), (T(n), n)]_H= \big( [T(m), T(n)],~ \rho(T(m))(n)\big).
\end{align*}
Thus the bracket in $Gr (T)$ is closed if and only if \eqref{def:et} holds. Hence the result follows.
\end{proof}

\begin{prop}\label{prop0}
Let  $(M,\fg, T)$ be an embedding tensor on Malcev algebra. Then we have a left Malcev dialgebra on ${M}$  with bracket defined by
$$[m, n]_{M}\triangleq \rho(T(m))(n),$$ 
for all $m,n\in M$,
\end{prop}

\begin{proof}
Using the identity \eqref{def:et}, we have
\begin{align*}
T([m,n]_{M})=T(\rho(T(m))(n))=[T(m),T(n)]_{\fg}.
\end{align*}
Thanks to \eqref{def:rep}, for $m, n, p, q\in M$, we get
\begin{align*}
&[[m,n]_{M},{p}]_{M} +[[n,m]_{M},{p}]_{M}\\
=&\rho \big(T([m,n]_{M})+T([n,m])_{M}\big)({p})\\
=&\rho \big([T(m),T(n)]_{\fg}+[T(n),T(m)]_{\fg}\big)({p})\\
=&0
\end{align*}
and
\begin{align*}
&[m,[n,[{p},{q}]_{M}]_{M}]_{M}-[[n,[{p},[m,{q}]_{M}]_{M}-[{p},[[m,n]_{M},{q}]_{M}]_{M}-[[m,{p}]_{M},[n,{q}]_{M}]_{M} \\
&- [[m,[n,{p}]_{M}]_{M},{q}]_{M}\\
=&\rho(T(m))\rho(T(n))\rho(T({p}))({q})-\rho(T(n))\rho(T({p}))\rho(T(m))({q})-\rho(T({p}))\rho(T([m,n]_{M}))({q})\\
&-\rho(T([m,{p}]_{M}))\rho(T(n))({q})-\rho(T([m,[n,{p}]_{M}]_{M}))({q})\\
=&\rho(T(m))\rho(T(n))\rho(T({p}))({q})-\rho(T(n))\rho(T({p}))\rho(T(m))({q})-\rho(T({p}))\rho([T(m),T(n)]_{\fg})({q})\\
&-\rho([T(m),T({p})]_{\fg})\rho(T(n))({q})-\rho([T(m),[T(n),T({p})]_{\fg}]_{\fg})({q})\\
=&0.
\end{align*}
Thus we obtain that $(M, [~,~]_{M})$ is a left Malcev dialgebra since $M$ is a representation of $\fg$.
This complete the proof.
\end{proof}

Similar computations shows that $M$  is a right Malcev dialgebra if we define the bracket by
$ [m, n]_{M}\triangleq \rho(T(n))(m)$.

\section{Representation}
In this section, the representation theory of an embedding tensor on Malcev algebras is given.

\begin{definition}
 Let $\big(M, \fg, T \big)$ be an embedding tensor on Malcev algebra. A representation $\left(\rho_{1}, \rho_{2}, \rho_{3}\right)$ of $\big(M, \fg, T \big)$ is an object $\big(V, W , T' \big)$ such that the following conditions are satisfied:

(i) $\rho_{1}$ and $\rho_{2}$ are representations of $\fg$ on $V$ and  $W$ respectively;

(ii) $T'$ is an equivariant map with respect to $\rho_{1}$ and $\rho_{2}$:
\begin{eqnarray}\label{deflm01}
T'\circ\rho_{1}(x)(v)=\rho_{2}(x)\circ T'(v);
\end{eqnarray}

(iii) there exists a map $\rho_{3}: M \rightarrow \operatorname{Hom}(W, V)$ such that
\begin{align}\label{deflm02}
\rho_{2}(T(m))(T'(v))=\rho_{3}(m)(T'(v)),
\end{align}
where $m \in M$ and $v \in V$;
\newpage
(iv) the following compatibility conditions are satisfied:
\begin{align}\label{deflm03}
\rho_{3}(m)\rho_{2}(x)\rho_{2}(y)(w)-\rho_{1}(x)\rho_{1}(y)\rho_{3}(m)(w)+\rho_{1}(y)\rho_{3}(\rho(x)(m))(w)\nonumber\\
+\rho_{3}(\rho(y)(m))\rho_{2}(x)(w)+\rho_{3}(\rho([x,y])(m))(w)=0,
\end{align}
\begin{align}\label{deflm04}
\rho_{3}(m) \rho_{2}\big([x,y]\big)(w_{})-\rho_{1}(x)\rho_{1}(y)\rho_{3}(m)(w_{})+\rho_{3}\big(\rho(x)\rho(y)(m)\big)(w_{})\nonumber\\
+\rho_{1}(y)\rho_{3}(m)\rho_{2}(x)(w_{})+\rho_{3}\big(\rho(x)(m)\big)\rho_{2}(y)(w_{})=0,
\end{align}
\begin{align}\label{deflm05}
\rho_{3}(m) \rho_{2}(x)\rho_{2}(y)(w_{})-\rho_{3}\big(\rho(y)\rho(x)(m)\big)(w_{})+\rho_{1}(x)\rho_{3}(\rho(y)(m))(w_{})\nonumber\\
-\rho_{1}(y)\rho_{3}(m)\rho_{2}(x)(w_{})+\rho_{1}([y,x])\rho_{3}(m)(w_{})=0,
\end{align}
where $x, y \in \fg, w \in W$ and $m \in M$.
\end{definition}
We construct semidirect products of an embedding tensor on Malcev algebra structure using its representations.
\begin{prop}
Given a representation of an embedding tensor $(M, \fg, T)$ on $(V, W, T')$. Define on $(M \oplus V, \fg \oplus W, T+T')$ an embedding tensor  by
\begin{align}
\left\{\begin{aligned}
\widehat{T}(m+v) & \triangleq T(m)+T'(v), \\
{\left[x+w_{1}, y+w_{2} \right]_{\widehat{g}} } & \triangleq\left[x, y\right]+\rho_{2}(x)\left(w_{2}\right)-\rho_{2}\left(y\right)(w_{1}), \\
\widehat{\rho}(x+w)(m+v) & \triangleq \rho(x)(m)+\rho_{1}(x)(v)-\rho_{3}(m)(w) .
\end{aligned}\right.
\end{align}
This is called the semidirect product of the embedding tensor $(M, \fg, T)$ and $(V, W, T')$.
\end{prop}

\begin{proof}
Consider the Malcev algebra $\fg \oplus W$ together with bracket given by
\begin{align*}
{\left[x+w_{1}, y+w_{2} \right] } & \triangleq\left[x, y\right]+\rho_{2}(x)\left(w_{2}\right)-\rho_{2}\left(y\right)(w_{1}).
\end{align*}
We are going to prove that $(M \oplus V, \fg \oplus W, T+T')$ is an embedding tensor on Malcev algebra $\fg \oplus W$.
First we prove that $\widehat{T}$ is indeed an embedding tensor operator:
\begin{align}\label{semidirect01}
[\widehat{T}(m+u),\widehat{T}(n+v)]=\widehat{T} \big(\widehat{\rho}(\widehat{T}(m+u))(n+v)\big).
\end{align}
The left hand side of \eqref{semidirect01} is
\begin{align*}
&[\widehat{T}(m+u),\widehat{T}(n+v)]\\
=&[T(m)+T'(u),T(n)+T'(v)]\\
=&[T(m),T(n)]+\rho_{2}(T(m))(T'(v))-\rho_{2}(T(n))(T'(u)),
\end{align*}
and the right hand side of \eqref{semidirect01} is
\begin{align*}
&\widehat{T} \big(\widehat{\rho}(\widehat{T}(m+u))(n+v)\big)\\
=&\widehat{T} \big(\widehat{\rho}(T(m)+T'(u))(n+v)\big)\\
=&\widehat{T} \big(\rho(T(m))(n)+\rho_{1}(T(m))(v)-\rho_{3}(n)(T'(u))\big)\\
=&T \big(\rho(T(m))(n)\big)+T' \big(\rho_{1}(T(m))(v)-\rho_{3}(n)(T'(u))\big).
\end{align*}
Thus the two sides of \eqref{semidirect01} are equal if and only if the compatibility condition \eqref{deflm01} and  \eqref{deflm02}   holds.

Next we prove that $\widehat{\rho}$ is indeed a representation of $\fg \oplus W$ on $M \oplus V$:
\begin{align}\label{semidirect02}
&\widehat{\rho}\big(\big[[x+w_{1}, y+w_{2}], z+w_{3}\big]\big)(m+v) \nonumber \\
=&\widehat{\rho}(x+w_{1})\widehat{\rho}(y+w_{2})\widehat{\rho}(z+w_{3})(m+v)
-\widehat{\rho}(z+w_{3})\widehat{\rho}(x+w_{1})\widehat{\rho}(y+w_{2})(m+v) \nonumber \\
&+\widehat{\rho}(y+w_{2})\widehat{\rho}\big([z+w_{3},x+w_{1}]\big)(m+v)
-\widehat{\rho}\big([y+w_{2},z+w_{3}]\big)\widehat{\rho}(x+w_{1})(m+v).
\end{align}
The left hand side of  \eqref{semidirect02} is
$$
\begin{aligned}
&\widehat{\rho}\big(\big[[x+w_{1}, y+w_{2}], z+w_{3}\big]\big)(m+v)\\
=& \widehat{\rho} \big(\big[[x,y]+ \rho_{2}(x)(w_{2})-\rho_{2}(y)(w_{1}),z+w_{3}\big]\big)(m+v)\\
=& \widehat{\rho}\bigg( \big(\big[[x,y], z\big]\big)+\rho_{2}\big([x,y]\big)(w_{3})-\rho_{2}(z)\rho_{2}(x)(w_{2})+\rho_{2}(z)\rho_{2}(y)(w_{1}) \bigg)(m+v)\\
=&\rho \big(\big[[x,y], z\big]\big)(m)+ \rho_{1} \big(\big[[x,y], z\big]\big)(v)-\rho_{3}(m) \rho_{2}\big([x,y]\big)(w_{3})+\rho_{3}(m)\rho_{2}(z)\rho_{2}(x)(w_{2})\\
&-\rho_{3}(m) \rho_{2}(z)\rho_{2}(y)(w_{1}),
\end{aligned}
$$
and the right hand side of   \eqref{semidirect02}  is
$$
\begin{aligned}
&\widehat{\rho}(x+w_{1})\widehat{\rho}(y+w_{2})\widehat{\rho}(z+w_{3})(m+v)-\widehat{\rho}(z+w_{3})\widehat{\rho}(x+w_{1})\widehat{\rho}(y+w_{2})(m+v)\\
&+\widehat{\rho}(y+w_{2})\widehat{\rho}\big([z+w_{3},x+w_{1}]\big)(m+v)-\widehat{\rho}\big([y+w_{2},z+w_{3}]\big)\widehat{\rho}(x+w_{1})(m+v)\\
=&\rho(x)\rho(y)\rho(z)(m)+\rho_{1}(x)\rho_{1}(y)\rho_{1}(z)(v)-\rho_{1}(x)\rho_{1}(y)\rho_{3}(m)(w_{3})-\rho_{1}(x)\rho_{3}(\rho(z)(m))(w_{2})\\
&-\rho_{3}\big(\rho(y)\rho(z)(m)\big)(w_{1})-\rho(z)\rho(x)\rho(y)(m)-\rho_{1}(z)\rho_{1}(x)\rho_{1}(y)(v)+\rho_{1}(z)\rho_{1}(x)\rho_{3}(m)(w_{2})\\
&+\rho_{1}(z)\rho_{3}(\rho(y)(m))(w_{1})+\rho_{3}\big(\rho(x)\rho(y)(m)\big)(w_{3})+\rho(y)\rho([z,x])(m)+\rho_{1}(y)\rho_{1}([z,x])(v)\\
&-\rho_{1}(y)\rho_{3}(m)\rho_{2}(z)(w_{1})+\rho_{1}(y)\rho_{3}(m)\rho_{2}(x)(w_{3})-\rho_{3}\big(\rho([z,x])(m)\big)(w_{2})-\rho([y,z])\rho(x)(m)\\
&-\rho_{1}([y,z])\rho_{1}(x)(v)+\rho_{1}([y,z])\rho_{3}(m)(w_{1})+\rho_{3}\big(\rho(x)(m)\big)\rho_{2}(y)(w_{3})
-\rho_{3}\big(\rho(x)(m)\big)\rho_{2}(z)(w_{2}).
\end{aligned}
$$
Thus the two sides of   \eqref{semidirect02} are equal to each other since $\rho, \rho_{1}$ are representation of $\fg$ on $M, V$ and the compatibility condition \eqref{deflm03}, \eqref{deflm04} and \eqref{deflm05}  hold.
This complete the proof.
\end{proof}

\section{Abelian extensions}
Abelian extensions of Malcev algebras was studied in \cite{Yam2}.
In this section we study abelian extensions of embedding tensor on Malcev algebras using the second cohomology group.

\begin{definition} Let $(M, \fg, T,)$ and $(V, W, T')$ be two embedding tensors on Malcev algebra. An abelian  extension of $(M, \fg, T,)$ by $(V, W, T')$ is  a  short exact sequence
\begin{equation}
\xymatrix{
   0  \ar[r]^{} & {V} \ar[d]_{T'} \ar[r]^{i_{0}} & {\widehat{M}} \ar[d]_{\widehat{T}} \ar[r]^{p_{0}} &{M} \ar[d]_{T} \ar[r]^{} & 0 \\
   0 \ar[r]^{} & {W} \ar[r]^{i_{1}} & \widehat{\fg} \ar[r]^{p_{1}} &{\fg} \ar[r]^{} & 0
   }
\end{equation}
where $(\widehat{M}, \widehat{\fg}, \widehat{T},)$ is an embedding tensor on Malcev algebra.
We call $(\widehat{M}, \widehat{\fg}, \widehat{T},)$ an extension of  $(M, \fg, T,)$ by $(V, W, T')$, and denote it by $\widehat{E}$. It is an abelian extension if $(V, W, T')$ is an abelian embedding tensor on Malcev algebra (this means that the bracket on $W$ and $V$ are zero, the representation of $W$ on $V$ is trivial: for $\forall w \in W, \forall v \in V$, we have $\widehat{\rho}(w)(v)=0$).
\end{definition}

A splitting $\sigma=\left(\sigma_{0}, \sigma_{1}\right):(M, \fg, T,) \rightarrow(\widehat{M}, \widehat{\fg}, \widehat{T},)$ consists of linear maps $\sigma_{1}: \fg \rightarrow \widehat{\fg}$ and $\sigma_{0}: M \rightarrow \widehat{M}$ such that $p_{0} \circ \sigma_{0}=\mathrm{id}_{M}, p_{1} \circ \sigma_{1}=\mathrm{id}_{\fg}$ and $\widehat{T} \circ \sigma_{0}=\sigma_{1} \circ T$.

Two extensions of Lie algebras $\widehat{\mathrm{E}}: 0 \longrightarrow(V, W, T') \stackrel{i}{\longrightarrow}(\widehat{M}, \widehat{\fg}, \widehat{T},) \stackrel{p}{\longrightarrow}(M, \fg, T,) \longrightarrow 0$ and $\widetilde{\mathrm{E}}: 0 \longrightarrow(V, W, T') \stackrel{j}{\longrightarrow}(\widetilde{M}, \widetilde{\fg}, \widetilde{T}) \stackrel{q}{\longrightarrow}(M, \fg, T) \longrightarrow 0$ are equivalent, if there exists a morphism $F:(\widehat{M}, \widehat{\fg}, \widehat{T}) \longrightarrow(\widetilde{M}, \widetilde{\fg}, \widetilde{T})$ such that $F \circ i=j, q \circ F=p$

Let $(\widehat{M}, \widehat{\fg}, \widehat{T})$ be an extension of $(M, \fg, T)$ by $(V, W, T')$ and $\sigma:(M, \fg, T) \rightarrow(\widehat{M}, \widehat{\fg}, \widehat{T})$ be a splitting. Define $\rho_{1}: \fg \rightarrow \mathfrak{gl}(V)$,  $\rho_{2}: \fg \rightarrow \mathfrak{gl}(W)$ and $\rho_{3}: M \rightarrow \Hom(W,V)$ by
$$
\left\{\begin{aligned}
\rho_{1}(x)(v)=&\widehat{\rho}(\sigma_{1}(x))(v),\\
\rho_{2}(x)(w)=&[\sigma_{1}(x),w]_{\mathfrak{\widehat{g}}},\\
\rho_{3}(m)(w)=&-\widehat{\rho}(w)\sigma_{0}(m),\\
\end{aligned}\right.
$$
for all $x, y, z \in \fg, v \in V$ and $m \in M$.

\begin{prop}
 With the above notations, $\left(\rho_{1}, \rho_{2}, \rho_{3}\right)$ is a representation of $(M, \fg, T)$ on $(V, W, T')$. Furthermore, $\rho_i$ does not depend on the choice of the splitting $\sigma$. Moreover, equivalent abelian extensions give the same representation.
\end{prop}

\begin{proof}
Firstly, we show that $\left(\rho_{1}, \rho_{2}, \rho_{3}\right)$ is well-defined. Since $\operatorname{Ker} p_{0} \cong W$, then for $w \in W$, we have $p_{0}(w)=0$. By the fact that $\left(p_{0}, p_{1}\right)$ is a homomorphism between $(\widehat{M}, \widehat{\fg}, \widehat{T})$ and $(M, \fg, T)$, we get
$$
p_{1}\left[\sigma_{1}(x), w\right]_{\widehat{\fg}}=\left[p_{1} \sigma_{1}(x), p_{1}(w)\right]_{\widehat{\fg}}=\left[p_{1} \sigma_{1}(x), 0\right]_{\widehat{\fg}}=0.
$$
Thus $\rho_{2}(x)(w)=\left[\sigma_{1}(x), w\right]_{\widehat{\fg}} \in \operatorname{ker} p_{1} \cong W$. Similar computations show that $\rho_{1}(x)(v)=$ $\widehat{\rho}\left(\sigma_{1}(x)\right)(v), \rho_{3}(v)(w)=-\widehat{\rho}(w)\left(\sigma_{0}(w)\right) \in \operatorname{Ker} p_{0}=V$. Now we will show that $\rho_{i}$ are independent of the choice of $\sigma$. In fact, if we choose another splitting $\sigma^{\prime}: \fg \rightarrow \widehat{\fg}$, then $p_{0}\left(\sigma_{0}(m)-\sigma_{0}^{\prime}(m)\right)=m-m=0, p_{1}\left(\sigma_{1}(x)-\sigma_{1}^{\prime}(x)\right)=x-x=0$, i.e. $\sigma_{0}(v)-\sigma_{0}^{\prime}(v) \in \operatorname{Ker} p_{0}=V$, $\sigma_{1}(x)-\sigma_{1}^{\prime}(x) \in \operatorname{Ker} p_{1}=W$. Thus, $\left[\sigma_{0}(m)-\sigma_{0}^{\prime}(m), w\right]_{\widehat{\fg}}=0,\left[\sigma_{1}(x)-\sigma_{1}^{\prime}(x), w\right]_{\widehat{\fg}}=0$, which implies that $\rho_{1}, \rho_{2}, \rho_{3}$ are independent on the choice of $\sigma$. Thus $\rho_{1}, \rho_{2}, \rho_{3}$ are well-defined.

Secondly, we check that $\left(\rho_{1}, \rho_{2}, \rho_{3}\right)$ is indeed a representation of $(M, \fg, T)$ on $(M, W, T')$. Since $(V, W, T')$ is an abelian embedding tensor on Malcev algebra, we have
\begin{align*}
&\rho_{2}([[x, y], z])(w)-\rho_{2}(x) \rho_{2}(y) \rho_{2}(z)(w)+\rho_{2}(z) \rho_{2}(x) \rho_{2}(y)(w)\\
&-\rho_{2}(y) \rho_{2}([z, x])(w)+\rho_{2}([y, z]) \rho_{2}(x)(w)\\
=&\bigg[\sigma_{1}(\big[[x, y], z\big]), w \bigg]-\bigg[\sigma_{1}(x),\big[\sigma_{1}(y),[\sigma_{1}(z),w]]\big]\bigg]+\bigg[\sigma_{1}(z),\big[\sigma_{1}(x),[\sigma_{1}(y),w]]\big]\bigg]\\
&-\bigg[\sigma_{1}(y),\big[\sigma_{1}([z, x]), w\big]\bigg]+\bigg[\sigma_{1}([y,z]),\big[\sigma_{1}(x), w\big]\bigg]\\
=&\bigg[\big[\sigma_{1}([x, y]), \sigma_{1}(z)\big], w \bigg]-\bigg[\sigma_{1}(x),\big[\sigma_{1}(y),[\sigma_{1}(z),w]]\big]\bigg]+\bigg[\sigma_{1}(z),\big[\sigma_{1}(x),[\sigma_{1}(y),w]]\big]\bigg]\\
&-\bigg[\sigma_{1}(y),\big[[\sigma_{1}(z), \sigma_{1}(x)], w\big]\bigg]+\bigg[[\sigma_{1}(y),\sigma_{1}(z)],\big[\sigma_{1}(x), w\big]\bigg]\\
=&\bigg[\big[[\sigma_{1}(x), \sigma_{1}(y)], \sigma_{1}(z)\big], w \bigg]-\bigg[\sigma_{1}(x),\big[\sigma_{1}(y),[\sigma_{1}(z),w]]\big]\bigg]+\bigg[\sigma_{1}(z),\big[\sigma_{1}(x),[\sigma_{1}(y),w]]\big]\bigg]\\
&-\bigg[\sigma_{1}(y),\big[[\sigma_{1}(z), \sigma_{1}(x)], w\big]\bigg]+\bigg[[\sigma_{1}(y),\sigma_{1}(z)],\big[\sigma_{1}(x), w\big]\bigg]\\
=&0,
\end{align*}
which implies that
\begin{align}
&\rho_{2}([[x, y], z])(w)-\rho_{2}(x) \rho_{2}(y) \rho_{2}(z)(w)+\rho_{2}(z) \rho_{2}(x) \rho_{2}(y)(w) \nonumber \\
&-\rho_{2}(y) \rho_{2}([z, x])(w)+\rho_{2}([y, z]) \rho_{2}(x)(w)=0.
\end{align}
Similarly, we have
\begin{align}
&\rho_{1}([[x, y], z])(w)-\rho_{1}(x) \rho_{1}(y) \rho_{1}(z)(w) +\rho_{1}(z) \rho_{1}(x) \rho_{1}(y)(w) \nonumber \\
&-\rho_{1}(y) \rho_{1}([z, x])(w)+\rho_{1}([y, z]) \rho_{1}(x)(w)=0.
\end{align}
For the equivalence between $\rho_{1}$ and $\rho_{2}$, we have
\begin{align}
T' \circ \rho_{1}(x)(v)=T' \circ \widehat{\rho}\left(\sigma_{1}(x)\right)(v)=\left[\sigma_{1}(x), T'(a)\right]=\rho_{2}(x) \circ T'(v).
\end{align}
For $\rho_{3}: M\rightarrow \Hom(W,V)$, we have
\begin{align}
T' \circ \rho_{3}(m)(w) & =-T' \circ \widehat{\rho}(w)\left(\sigma_{0}(m)\right)=-\left[w, \widehat{T} \sigma_{0}(m)\right] \nonumber \\
& =-\left[w, \sigma_{1}(T(m))\right]=\left[\sigma_{1}(T(m)), w\right]=\rho_{2}(T(m))(w) .
\end{align}
By the fact that $\left(p_{0}, p_{1}\right)$ is a homomorphism between $(\widehat{M}, \widehat{\fg}, \widehat{T})$ and $(M, \fg, T)$, we have
$$
\widehat{\rho}(\sigma_{1}(x))\sigma_{0}(m)-\sigma_{0}(\rho(x)(m)) \in V
$$
One also check that
\begin{align*}
&\rho_{3}(m)\rho_{2}(x)\rho_{2}(y)(w)-\rho_{1}(x)\rho_{1}(y)\rho_{3}(m)(w)+\rho_{1}(y)\rho_{3}(\rho(x)(m))(w)\nonumber\\
&+\rho_{3}(\rho(y)(m))\rho_{2}(x)(w)+\rho_{3}(\rho([x,y])(m))(w)\\
=&-\widehat{\rho}\big([\sigma_{1}(x),[\sigma_{1}(y),w]]\big)(\sigma_{0}(m))+\widehat{\rho}(\sigma_{1}(x))\widehat{\rho}(\sigma_{1}(y))
\widehat{\rho}(w)(\sigma_{0}(m))\\
&-\widehat{\rho}(\sigma_{1}(y))\widehat{\rho}(w)(\sigma_{0}(\rho(x)(m)))
-\widehat{\rho}([\sigma_{1}(x),w])(\sigma_{0}(\rho(y)(m)))-\widehat{\rho}(w)(\sigma_{0}(\rho([x,y])(m)))\\
=&-\widehat{\rho}\big([\sigma_{1}(x),[\sigma_{1}(y),w]]\big)(\sigma_{0}(m))+\widehat{\rho}(\sigma_{1}(x))\widehat{\rho}(\sigma_{1}(y))
\widehat{\rho}(w)(\sigma_{0}(m))\\
&-\widehat{\rho}(\sigma_{1}(y))\widehat{\rho}(w)\widehat{\rho}(\sigma_{1}(x))(\sigma_{0}(m))
-\widehat{\rho}([\sigma_{1}(x),w])\widehat{\rho}(\sigma_{1}(y))(\sigma_{0}(m))\\
&-\widehat{\rho}(w)\widehat{\rho}([\sigma_{1}(x),\sigma_{1}(y)])(\sigma_{0}(m))\\
=&0,
\end{align*}
which implies that
\begin{align*}
\rho_{3}(m)\rho_{2}(x)\rho_{2}(y)(w)-\rho_{1}(x)\rho_{1}(y)\rho_{3}(m)(w)+\rho_{1}(y)\rho_{3}(\rho(x)(m))(w)\nonumber\\
+\rho_{3}(\rho(y)(m))\rho_{2}(x)(w)+\rho_{3}(\rho([x,y])(m))(w)=0.
\end{align*}
Similarly, we have
\begin{align*}
\rho_{3}(m) \rho_{2}\big([x,y]\big)(w_{})-\rho_{1}(x)\rho_{1}(y)\rho_{3}(m)(w_{})+\rho_{3}\big(\rho(x)\rho(y)(m)\big)(w_{})\nonumber\\
+\rho_{1}(y)\rho_{3}(m)\rho_{2}(x)(w_{})+\rho_{3}\big(\rho(x)(m)\big)\rho_{2}(y)(w_{})=0,\\
\\
\rho_{3}(m)\rho_{2}(x)\rho_{2}(y)(w_{})-\rho_{3}\big(\rho(y)\rho(x)(m)\big)(w_{})+
\rho_{1}(x)\rho_{3}(\rho(y)(m))(w_{})\nonumber\\
-\rho_{1}(y)\rho_{3}(m)\rho_{2}(x)(w_{})+\rho_{1}([y,x])\rho_{3}(m)(w_{})=0.
\end{align*}
Therefore, $\left(\rho_{1}, \rho_{2}, \rho_{3}\right)$ is a representation of $(M, \fg, T)$.

Finally, suppose that $\widehat{\mathrm{E}}$ and $\widetilde{\mathrm{E}}$ are equivalent abelian extensions, and $F:(\widehat{M}, \widehat{\fg}, \widehat{T}) \longrightarrow$ $(\widetilde{M}, \widetilde{\fg}, \widetilde{T})$ be the morphism. Choose linear sections $\sigma$ and $\sigma^{\prime}$ of $p$ and $q$. Then we have $q_{1} F_{1} \sigma_{1}(x)=p_{1} \sigma_{1}(x)=x=q_{1} \sigma_{1}^{\prime}(x)$, thus $F_{1} \sigma_{1}(x)-\sigma_{1}^{\prime}(x) \in \operatorname{Ker} q_{1}=W$. Therefore, we obtain
$$
\left[\sigma_{1}^{\prime}(x), v+w\right]_{\widetilde{\fg}}=\left[F_{1} \sigma_{1}(x), v+w\right]_{\widetilde{\fg}}=F_{1}\left[\sigma_{1}(x), v+w\right]_{\hat{\fg}}=\left[\sigma_{1}(x), v+w\right]_{\hat{\fg}},
$$
which implies that equivalent abelian extensions give the same $\rho_{1}$ and $\rho_{2}$. Similarly, we can show that equivalent abelian extensions also give the same $\rho_{3}$. Therefore, equivalent abelian extensions also give the same representation. The proof is finished.

\end{proof}
Let $\sigma:(M, \fg, T) \rightarrow(\widehat{M}, \widehat{\fg}, \widehat{T})$ be a splitting of an abelian extension. Define the following linear maps:
\begin{equation}
\left\{\begin{array}{rlrlrl}
\theta: & M & \longrightarrow & W, ~~\qquad \theta(m) & \triangleq & \widehat{T} (\sigma_{0}(m))-\sigma_{1}(T(m)), \\
\omega: & \fg \wedge \fg & \longrightarrow & W, \qquad \omega(x, y) & \triangleq & \left[\sigma_{1}(x), \sigma_{1}(y)\right]_{\widehat{\fg}}-\sigma_{1}\left([x, y]_{\fg}\right), \\
\nu: & \fg \otimes M  & \longrightarrow &  V, ~\qquad \nu(x, m) & \triangleq & \widehat{\rho}\left(\sigma_{1}(x)\right) \sigma_{0}(m)-\sigma_{0}(\rho(x)(m)),
\end{array}\right.
\end{equation}
for all $x, y, z \in \fg, w \in W$ and $m \in M$.
\begin{theorem}
With the above notations, $(\theta, \omega, \nu)$ is a 2-cocycle of $(M, \fg, T)$ with coefficients in $(V, W, T')$ satisfying the conditions \eqref{2coc01}, \eqref{2coc02} and \eqref{2coc03} below.
\end{theorem}
\begin{proof}
Since $\widehat{T}$ is an embedding tensor operator, we have the equality
\begin{align*}
[\widehat{T}(\sigma_{0}(m)),\widehat{T}(\sigma_{0}(n))]_{\widehat{\fg}}
-\widehat{T}\big(\widehat{\rho}(\widehat{T}(\sigma_{0}(m)))(\sigma_{0}(n)\big)=0.
\end{align*}
Then we get that
\begin{align*}
&[\theta(m)+\sigma_{1}(T(m)),\theta(n)+\sigma_{1}(T(n))]_{\widehat{\fg}}=\widehat{T}\bigg(\widehat{\rho}\big(\theta(m)+\sigma_{1}(T(m))\big)\sigma_{0}(n)\bigg),\\
\\
&[\theta(m),\theta(n)]_{W}+[\theta(m),\sigma_{1}(T(n))]_{\widehat{\fg}}
+[\sigma_{1}(T(m)),\theta(n)]_{\widehat{\fg}}+[\sigma_{1}(T(m)),\sigma_{1}(T(n))]_{\widehat{\fg}}\\
=&\widehat{T}\bigg(\widehat{\rho}(\theta(m))(\sigma_{0}(n))+\widehat{\rho}\big(\sigma_{1}(T(m))\big)(\sigma_{0}(n)
)\bigg),\\
\\
&\rho_{2}(T(m))(\theta(n))-\rho_{2}(T(n))(\theta(m))+\omega(T(m),(T(n))+\sigma_{1} ([T(m),T(n)]_{\fg})\\
=&T' \big(\nu(T(m),n)\big)+\theta \big(\rho(T(m)(n))\big)+\sigma_{1} T \big(\rho(T(m))(n)\big)-T' \big(\rho_{3}(n)(\theta(m))\big).
\end{align*}
Thus we obtain
\begin{align}\label{2coc01}
&\rho_{2}(T(m))(\theta(n))-\rho_{2}(T(n))(\theta(m))+\omega(T(m),(T(n)) \nonumber \\
&- T' \big(\nu(T(m),n)\big)-\theta \big(\rho(T(m)(n))\big)+ T' \big(\rho_{3}(n)(\theta(m))\big)=0.
\end{align}
Since $\left(\widehat{\fg}, [~,~]_{\widehat{\fg}}\right)$ is an abelian extention of Malcev algebra $\fg$ through $W$, we get
\begin{align*}
&\big[[\sigma_{1}(x),\sigma_{1}(z)]_{\widehat{\fg}},[\sigma_{1}(y),\sigma_{1}(t)]_{\widehat{\fg}}\big]_{\widehat{\fg}}
-\big[\big[[\sigma_{1}(x),\sigma_{1}(y)]_{\widehat{\fg}},\sigma_{1}(z)\big]_{\widehat{\fg}},\sigma_{1}(t)\big]_{\widehat{\fg}}\\
&-\big[\big[[\sigma_{1}(y),\sigma_{1}(z)]_{\widehat{\fg}},\sigma_{1}(t)\big]_{\widehat{\fg}},
\sigma_{1}(x)\big]_{\widehat{\fg}}-\big[\big[[\sigma_{1}(z),\sigma_{1}(t)]_
{\widehat{\fg}},\sigma_{1}(x)\big]_{\widehat{\fg}},\sigma_{1}(y)\big]_{\widehat{\fg}}\\
&-\big[\big[[\sigma_{1}(t),\sigma_{1}(x)]_{\widehat{\fg}},\sigma_{1}(y)\big]_{\widehat{\fg}},
\sigma_{1}(z)\big]_{\widehat{\fg}}\\
=&[\omega(x,z),\omega(y,t)]+\rho_{2}([x,z])(\omega(y,t))-\rho_{2}([y,t])(\omega(x,z))+\omega([x,z],[y,t])
+\sigma_{1}\big([[x,z],[y,t]]\big)\\
&-\rho_{2}(t)\rho_{2}(z)(\omega(x,y))+\rho_{2}(t)\big(\omega([x,y],z)\big)-\omega([[x,y],z],t)-\sigma_{1}\big(\big[[[x,y],z],t\big]\big)\\
&-\rho_{2}(x)\rho_{2}(t)(\omega(y,z))+\rho_{2}(x)\big(\omega([y,z],t)\big)-\omega([[y,z],t],x)-\sigma_{1}\big(\big[[[y,z],t],x\big]\big)\\
&-\rho_{2}(y)\rho_{2}(x)(\omega(z,t))+\rho_{2}(y)\big(\omega([z,t],x)\big)-\omega([[z,t],x],y)-\sigma_{1}\big(\big[[[z,t],x],y\big]\big)\\
&-\rho_{2}(z)\rho_{2}(y)(\omega(t,x))+\rho_{2}(z)\big(\omega([t,x],y)\big)-\omega([[t,x],y],z)-\sigma_{1}\big(\big[[[t,x],y],z\big]\big)\\
=&\rho_{2}([x,z])(\omega(y,t))-\rho_{2}([y,t])(\omega(x,z))+\omega([x,z],[y,t])-\rho_{2}(t)\rho_{2}(z)(\omega(x,y))\\
&+\rho_{2}(t)\big(\omega([x,y],z)\big)-\omega([[x,y],z],t)-\rho_{2}(x)\rho_{2}(t)(\omega(y,z))+\rho_{2}(x)\big(\omega([y,z],t)\big)\\
&-\omega([[y,z],t],x)-\rho_{2}(y)\rho_{2}(x)(\omega(z,t))+\rho_{2}(y)\big(\omega([z,t],x)\big)-\omega([[z,t],x],y)\\
&-\rho_{2}(z)\rho_{2}(y)(\omega(t,x))+\rho_{2}(z)\big(\omega([t,x],y)\big)-\omega([[t,x],y],z)\\
=&~0.
\end{align*}
Then we obtain
\begin{align}\label{2coc02}
&\rho_{2}([x,z])(\omega(y,t))+\rho_{2}(t)\big(\omega([[x,y],z])\big)+\rho_{2}(x)\big(\omega([[y,z],t])\big)+
\rho_{2}(y)\big(\omega([[z,t],x])\big)
\nonumber \\
&+\rho_{2}(z)\big(\omega([[t,x],y])\big)-\rho_{2}([y,t])(\omega(x,z))-\rho_{2}(t)\rho_{2}(z)(\omega(x,y))-
\rho_{2}(x)\rho_{2}(t)(\omega(y,z)) \nonumber \\
&-\rho_{2}(y)\rho_{2}(x)(\omega(z,t))-\rho_{2}(z)\rho_{2}(y)(\omega(t,x))+\omega([x,z],[y,t])-\omega \big([[x,y],z],t\big) \nonumber \\
&-\omega \big([[y,z],t],x\big)-\omega \big([[z,t],y],x\big)-\omega \big([[t,x],y],z\big)=0.
\end{align}
Since $\left(\widehat{M}, \widehat{\rho}\right)$ is a representation of $\left(\widehat{\fg}, [~,~]_{\widehat{\fg}}\right)$, we have the equality
\begin{align*}
&\widehat{\rho}([[\sigma_{1}(x), \sigma_{1}(y)]_{\widehat{\fg}}, \sigma_{1}(z)]_{\widehat{\fg}})(\sigma_{0}(m))-\widehat{\rho}(\sigma_{1}(x)) \widehat{\rho}(\sigma_{1}(y)) \widehat{\rho}(\sigma_{1}(z))(\sigma_{0}(m))\\
&+\widehat{\rho}(\sigma_{1}(z)) \widehat{\rho}(\sigma_{1}(x)) \widehat{\rho}(\sigma_{1}(y))(\sigma_{0}(m))-\widehat{\rho}(\sigma_{1}(y)) \widehat{\rho}([\sigma_{1}(z), \sigma_{1}(x)]_{\widehat{\fg}})(\sigma_{0}(m))\\
&+\widehat{\rho}([\sigma_{1}(y), \sigma_{1}(z)]_{\widehat{\fg}}) \widehat{\rho}(\sigma_{1}(x))(\sigma_{0}(m))\\
=&\rho_{3}(m)\big(\rho_{2}(z)(\omega(x,y))\big)-\rho_{3}(m)\big(\omega([x,y],z)\big)+\nu([[x,y],z],m)\\
&+\sigma_{0}\big(\rho([[x,y],z])(m)\big)-\rho_{1}(x)\rho_{1}(y)(\nu(z,m))-\rho_{1}(x)\big(\nu(y,\rho(z)(m))\big)\\
&-\nu\big(x,\rho(y)\rho(z)(m)\big)-\sigma_{0}\big(\rho(x)\rho(y)\rho(z)(m)\big)+\rho_{1}(z)\rho_{1}(x)(\nu(y,m))\\
&+\rho_{1}(z)\big(\nu(x,\rho(y)(m))\big)+\nu\big(z,\rho(x)\rho(y)(m)\big)
+\sigma_{0}\big(\rho(z)\rho(x)\rho(y)(m)\big)\\
&+\rho_{1}(y)\big(\rho_{3}((m)\omega(z,x))\big)-\rho_{1}(y)\big(\nu([z,x],m)\big)-\nu\big(y,\rho([z,x])(m)\big)\\
&-\sigma_{0}\big(\rho(y)\rho([z,x])(m)\big)-\rho_{3}(\rho(x)(m))(\omega(y,z))+\rho_{1}([y,z])(\nu(x,m))\\
&+\nu([y,z],\rho(x)(m))+\sigma_{0}\big(\rho([y,z])\rho(x)(m)\big)\\
=&\rho_{3}(m)\big(\rho_{2}(z)(\omega(x,y))\big)-\rho_{3}(m)\big(\omega([x,y],z)\big)+\nu([[x,y],z],m)\\
&-\rho_{1}(x)\rho_{1}(y)(\nu(z,m))-\rho_{1}(x)\big(\nu(y,\rho(z)(m))\big)-\nu\big(x,\rho(y)\rho(z)(m)\big)\\
&+\rho_{1}(z)\rho_{1}(x)(\nu(y,m))+\rho_{1}(z)\big(\nu(x,\rho(y)(m))\big)+\nu\big(z,\rho(x)\rho(y)(m)\big)\\
&+\rho_{1}(y)\big(\rho_{3}((m)\omega(z,x))\big)-\rho_{1}(y)\big(\nu([z,x],m)\big)-\nu\big(y,\rho([z,x])(m)\big)\\
&-\rho_{3}(\rho(x)(m))(\omega(y,z))+\rho_{1}([y,z])(\nu(x,m))+\nu([y,z],\rho(x)(m))\\
=&~0.
\end{align*}
Thus, we get
\begin{align}\label{2coc03}
&\rho_{3}(m)\big(\rho_{2}(z)(\omega(x,y))\big)-\rho_{3}(m)\big(\omega([x,y],z)\big)+\nu([[x,y],z],m) \nonumber \\
&-\rho_{1}(x)\rho_{1}(y)(\nu(z,m))-\rho_{1}(x)\big(\nu(y,\rho(z)(m))\big)-\nu\big(x,\rho(y)\rho(z)(m)\big) \nonumber \\
&+\rho_{1}(z)\rho_{1}(x)(\nu(y,m))+\rho_{1}(z)\big(\nu(x,\rho(y)(m))\big)+\nu\big(z,\rho(x)\rho(y)(m)\big) \nonumber \\
&+\rho_{1}(y)\big(\rho_{3}((m)\omega(z,x))\big)-\rho_{1}(y)\big(\nu([z,x],m)\big)-\nu\big(y,\rho([z,x])(m)\big) \nonumber \\
&-\rho_{3}(\rho(x)(m))(\omega(y,z))+\rho_{1}([y,z])(\nu(x,m))+\nu([y,z],\rho(x)(m)).
\end{align}
Therefore, we obtain that $(\theta, \omega, \nu)$ is a 2-cocycle of $(M, \fg, T)$ with coefficients in $(V, W, T')$. This complete the proof.
\end{proof}

Now we define the embedding tensor on Malcev algebra structure on $(M \oplus V, \fg \oplus W, \widehat{T})$ using the 2-cocycle given above. More precisely, we have
\begin{align}\label{eq52}
\left\{\begin{aligned}
\widehat{T}(m+v) & \triangleq T(m)+\theta(m)+T'(v), \\
{\left[x+w, x^{\prime}+w^{\prime}\right] } & \triangleq\left[x, x^{\prime}\right]+\omega\left(x, x^{\prime}\right)+\rho_{2}(x)\left(w^{\prime}\right)-\rho_{2}\left(x^{\prime}\right)(w), \\
\widehat{\rho}(x+w)(m+v) & \triangleq \rho(x)(m)+\nu(x, m)+\rho_{1}(x)(v)-\rho_{3}(m)(w),
\end{aligned}\right.
\end{align}
for all $x, y, z,t \in \fg, v \in V, w \in W$ and $m \in M$. Thus any extension $\widehat{E}$ is isomorphic to
\begin{equation}
\xymatrix{
   0  \ar[r]^{} & {V} \ar[d]_{T'} \ar[r]^{i_{0}} & {M \oplus V} \ar[d]_{\widehat{T}} \ar[r]^{p_{0}} &{M} \ar[d]_{T} \ar[r]^{} & 0 \\
   0 \ar[r]^{} & {W} \ar[r]^{i_{1}} & {\fg\oplus W} \ar[r]^{p_{1}} &{\fg} \ar[r]^{} & 0
   }
\end{equation}
where the embedding tensor on Malcev algebra structure is given as in \eqref{eq52}.
\begin{theorem}
There is a one-to-one correspondence between equivalence classes of abelian extensions and the second cohomology group $\mathbf{H}^{2}((M, \fg, T),(V, W, T'))$ which is defined by the spaces of 2-cocycles mod 2-coboundaries satisfying \eqref{55},  \eqref{56} and  \eqref{57} below.
\end{theorem}
\begin{proof}
We have know from the above discussion that abelian extensions of embedding tensor on Malcev algebra are correspond to 2-cocycle and vice verse. Let ${E}^{\prime}$ be another abelian extension determined by the 2-cocycle $\left(\theta^{\prime}, \omega^{\prime}, \nu^{\prime}\right)$. We are going to show that ${E}$ and ${E}^{\prime}$ are equivalent if and only if 2-cocycles $(\theta, \omega, \nu)$ and $\left(\theta^{\prime}, \omega^{\prime}, \nu^{\prime}\right)$ are in the same cohomology class.

Since $F$ is an equivalence of abelian extensions, there exist two linear maps $b_{0}: M \longrightarrow V$ and $b_{1}: \fg \longrightarrow W$ such that
$$
F_{0}(m+v)=m+b_{0}(m)+v, \quad F_{1}(x+w)=x+b_{1}(x)+w.
$$
First, by the equality
$$
\widehat{T}^{\prime} \circ F_{0}(m)=F_{1} \circ \widehat{T}(m).
$$
we have
\begin{align}\label{55}
\theta(m)-\theta^{\prime}(m)=T' b_{0}(m)-b_{1}(f(m)).
\end{align}
Furthermore, we have
$$
F_{1}([x, y]+\omega(x, y))=\left[F_{1}(x), F_{1}(y)\right]^{\prime},
$$
which implies that
\begin{align}\label{56}
\omega(x, y)-\omega^{\prime}(x, y)=\rho_{1}(x) b_{1}(y)-\rho_{1}(y) b_{1}(x)-b_{1}([x, y]).
\end{align}
Similarly, by the equality
$$
F_{1}([x, m]+\nu(x, m))=\left[F_{1}(x), F_{0}(m)\right]^{\prime},
$$
we get
\begin{align}\label{57}
\nu(x, m)-\nu^{\prime}(x, m)=\rho_{1}(x) b_{0}(m)-\rho_{2}(m) b_{1}(x)-b_{0}([x, m]) .
\end{align}
By \eqref{55},  \eqref{56} and  \eqref{57}, we deduce that $(\psi, \omega, \nu)-\left(\psi^{\prime}, \omega^{\prime}, \nu^{\prime}\right)=D\left(b_{0}, b_{1}\right)$. Thus, they are in the same cohomology class.

Conversely, if $(\theta, \omega, \nu)$ and $\left(\theta^{\prime}, \omega^{\prime}, \nu^{\prime}\right)$ are in the same cohomology class, assume that $(\theta, \omega, \nu)-\left(\theta^{\prime}, \omega^{\prime}, \nu^{\prime}\right)=D\left(b_{0}, b_{1}\right)$. Then we define $\left(F_{0}, F_{1}\right)$ by
$$
F_{0}(m+v)=m+b_{0}(m)+v, \quad F_{1}(x+w)=x+b_{1}(x)+w.
$$
Similar as the above proof, we can show that $\left(F_{0}, F_{1}\right)$ is an equivalence. We omit the details.
\end{proof}

\section{Infinitesimal deformations}
In this section, we study infinitesimal deformations of embedding tensors on Malcev algebra. For infinitesimal  deformations of associative algebras and Lie algebras, see \cite{Ge,NR}.

Let $(M, \fg, T)$ be an embedding tensor and $\theta: M \rightarrow \fg, \omega: \wedge^{2} \fg \rightarrow \fg, \nu: \fg \otimes M \rightarrow M$ be linear maps. Consider a $\lambda$-parametrized family of linear operations:
$$
\begin{aligned}
T_{\lambda}(m) & \triangleq T(m)+\lambda \theta(m), \\
{[x, y]_{\lambda} } & \triangleq[x, y]+\lambda \omega(x, y), \\
\rho_{\lambda}(x)(m) & \triangleq \rho(x)(m)+\lambda \nu(x)(m) .
\end{aligned}
$$
If $\left(M_{\lambda}, \fg_{\lambda}, T_{\lambda}\right)$ forms an embedding tensor, then we say that $(\theta, \omega, \nu)$ generates a 1-parameter infinitesimal deformation of $(M, \fg, T)$.

\begin{theorem} With the notations above, $(\theta, \omega, \nu)$ generates a 1-parameter infinitesimal deformation of $(M, \fg, T)$ if and only if the following conditions hold:

(i) $(\theta, \omega, \nu)$ is a 2-cocycle of $(M, \fg, T)$ with coefficients in the adjoint representation (satisfying (16), (19), (23)) and the following conditions
 (17),  (20),   (24) hold;

(ii) $(M, \fg, \theta)$ is an embedding tensor on Malcev algebra structure with bracket $\omega$ on $\fg$ and action of $\fg$ on $M$ by $\nu$.
\end{theorem}

\begin{proof}
If $\left(M_{\lambda}, \fg_{\lambda}, T_{\lambda}\right)$ is an embedding tensor on Malcev algebra, then $T_{\lambda}$ is an embedding tensor operator.
Thus we have
\begin{align*}
& [T_{\lambda}(m), T_{\lambda}(n)]_{\lambda}-T_{\lambda}\big(\rho_{\lambda}(T_{\lambda}(m))(n)\big)\\
=&[T(m)+\lambda \theta(m), T(n)+\lambda \theta(n)]_{\lambda}-T_{\lambda}\big(\rho_{\lambda}(T(m)+\lambda \theta(m))(n)\big)\\
=&[T(m)+\lambda \theta(m), T(n)+\lambda \theta(n)]+\lambda \omega(T(m)+\lambda \theta(m), T(n)+\lambda \theta(n))\\
&-T_{\lambda}\big(\rho(T(m)+\lambda \theta(m))(n)+\lambda \nu(T(m)+\lambda \theta(m),n)\big)\\
=&[T(m), T(n)]+[T(m),\lambda \theta(n)]+[\lambda \theta(m),T(n)]+[\lambda \theta(m),\lambda \theta(n)]\\
&+\lambda \omega(T(m),T(n))+\lambda \omega(T(m),\lambda \theta(n))+\lambda \omega(\lambda \theta(m), T(n))+\lambda \omega(\lambda \theta(m), \lambda \theta(n))\\
&-T\big(\rho(T(m))(n)\big)-T\big(\rho(\lambda \theta(m))(n)\big)-T\big(\lambda \nu(T(m), n)\big)-T\big(\lambda \nu(\lambda \theta(m), n)\big)\\
&-\lambda \theta \big(\rho(T(m))(n)\big)-\lambda \theta \big(\rho(\lambda \theta(m))(n)\big)-\lambda \theta \big(\lambda \nu(T(m),n)\big)-\lambda \theta \big(\lambda \nu(\lambda \theta(m),n)\big),
\end{align*}
which implies that
\begin{align}
\label{(16)}
&[T(m), \theta(n)]+[\theta(m), T(n)]+\omega(T(m),T(n))-T\big(\rho(\theta(m))(n)\big) \nonumber \\
&-T\big(\nu(T(m),n)\big)-\theta \big(\rho(T(m))(n)\big)=0,\\
\label{(17)}
&[\theta(m), \theta(n)]+\omega(T(m),\theta(n))+\omega(\theta(m), T(n))-T\big(\nu(\theta(m),n)\big) \nonumber \\
&-\theta \big(\rho(\theta(m))(n)\big)-\theta \big(\nu(T(m),n)\big)=0,\\
\label{(18)}
&\omega(\theta(m),\theta(n))-\theta \big(\nu(\theta(m),n)\big)=0.
\end{align}
Since $\fg_{\lambda}$ is a Malcev algebra, 
then we have
\begin{align*}
&\big[[x,z]_{\lambda},[y,t]_{\lambda}\big]_{\lambda}-\big[\big[[x,y]_{\lambda},z\big]_{\lambda},t\big]_{\lambda}-
\big[\big[[y,z]_{\lambda},t\big]_{\lambda},x\big]_{\lambda}\\
&-\big[\big[[z,t]_{\lambda},x\big]_{\lambda},y\big]_{\lambda}-\big[\big[[t,x]_{\lambda},y\big]_{\lambda},z\big]_{\lambda}=0,
\end{align*}
which implies that
\begin{align}
 &\omega([x, z],[y, t])-\omega([[x, y], z], t)-\omega([[y, z], t], x)-\omega([[z, t], x], y) \nonumber \\
&+\omega([[t, x], y], z)+[\omega(x, z),[y, t]]-[\omega([x, y], z), t]-[\omega([y, z], t), x] \nonumber \\
&-[\omega([z, t], x), y]-[\omega([t, x], y), z]+[[x, z], \omega(y, t)]-[[\omega(x, y), z], t] \nonumber \\
\label{(19)}&-[[\omega(y, z), t], x]-[[\omega(z, t), x], y]-[[\omega(t, x), y], z]=0, \\[1em]
 &\omega(\omega(x, z),[y, t])-\omega([\omega(x, y), z], t)-\omega([\omega(y, z), t], x)-\omega(\omega([z, t], x), y)\nonumber \\
&+\omega(\omega([t, x], y), z) +\omega([x, z], \omega(y, t))-\omega([\omega(x, y), z], t)-\omega([\omega(y, z), t], x)\nonumber \\
&-\omega([\omega(z, t), x], y)+\omega([\omega(t, x), y], z)+[\omega(x, z), \omega(y, t)]-[\omega(\omega(x, y), z), t] \nonumber \\
\label{(20)}&-[\omega(\omega(y, z), t), x]-[\omega(\omega(z, t), x), y]-[\omega(\omega(t, x), y), z] =0, \\[1em]
&\omega(\omega(x, z), \omega(y, t))-\omega(\omega(\omega(x, y), z), t)-\omega(\omega(\omega(y, z), t), x)-\omega(\omega(\omega(z, t), x), y) \nonumber \\
\label{(21)}&-\omega(\omega(\omega(t, x), y), z)=0.
\end{align}
Since $\left(M_{\lambda}, \rho_{\lambda}\right)$ is a representation of $\fg_{\lambda}$, we have
\begin{align*}
&\rho_{\lambda}([[x, y], z])(m)-\rho_{\lambda}(x) \rho_{\lambda}(y) \rho_{\lambda}(z)(m)+\rho_{\lambda}(z) \rho_{\lambda}(x) \rho_{\lambda}(y)(m)\nonumber \\
&-\rho_{\lambda}(y) \rho_{\lambda}([z, x])(m) +\rho_{\lambda}([y, z]) \rho_{\lambda}(x)(m)=0,
\end{align*}
which implies that
\begin{align}
&\nu([[x, y], z], m)-\nu(x, \rho(y) \rho(z)(m))+\nu(z, \rho(x) \rho(y)(m))-\nu(y, \rho([z, x])(m)) \nonumber \\
&+\nu([y, z], \rho(x)(m))+\rho(\omega([x, y], z))(m)-\rho(x) \nu(y, \rho(z)(m))+\rho(z) \nu(x, \rho(y)(m)) \nonumber \\
&-\rho(y) \nu([z, x], m)+\rho(\omega(y, z)) \rho(x)(m)+\rho([\omega(x, y), z])(m)-\rho(x) \rho(y) \nu(z, m)\nonumber\\
\label{(23)}&+ \rho(z) \rho(x) \nu(y, m)-\rho(y) \rho(\omega(z, x))(m)+\rho([y, z]) \nu(x, m)=0,\\[1em]
&\nu(\omega([x, y], z), m)-\nu(x, \nu(y, \rho(z)(m)))+\nu(z, \nu(x, \rho(y)(m)))-\nu(y, \nu([z, x],m)) \nonumber \\
&+\nu(\omega(y, z), \rho(x)(m))+\nu([\omega(x, y), z],m)-\nu(x, \rho(y) \nu(z, m))+\nu(z, \rho(x) \nu(y, m))\nonumber\\
&-\nu(y, \rho(\omega(z, x))(m))+\nu([y, z],\nu(x,m))+\rho(\omega(\omega(x, y), z))(m)-\rho(x) \nu(y, \nu(z,m)) \nonumber \\
\label{(24)}&+\rho(z) \nu(x, \nu(y,m))-\rho(y) \nu(\omega(z, x),m)+\rho(\omega(y, z)) \nu(x, m)=0,\\[1em]
&\nu(\omega(\omega(x, y), z), m)-\nu(x, \nu(y, \nu(z, m)))+\nu(z, \nu(x, \nu(y,m))) \nonumber \\
\label{(25)}&-\nu(y, \nu(\omega(z, x), m))+\nu(\omega(y, z), \nu(x,m))=0.
\end{align}
By \eqref{(16)}, \eqref{(17)}, \eqref{(19)}, \eqref{(20)}, \eqref{(23)} and \eqref{(24)}, we find that $(\theta, \omega, \nu)$ is a 2-cocycle of $(M, \fg, T)$ with the coefficients in the adjoint representation. Furthermore, by \eqref{(18)},\eqref{(21)} and \eqref{(25)}, $(M, \fg, \theta)$ with bracket $\omega$ and $\nu$ is an embedding tensor on Malcev algebra.
\end{proof}

Now we introduce the notion of Nijenhuis operators which gives trivial deformations.

Let $(M, \fg, T)$ be an embedding tensor on Malcev algebra and $N=\left(N_{0}, N_{1}\right)$ be a pair of linear maps $N_{0}: M \rightarrow M $ and $N_{1}: \fg \rightarrow \fg$ such that $T \circ N_{0}=N_{1} \circ T$. Define an exact 2-cochain
$$
(\omega, \nu, \theta)=D\left(N_{0}, N_{1}\right).
$$
by differential $D$ discussed above, i.e.,
$$
\begin{aligned}
\theta(m) & =T \circ N_{0}(m)-N_{1} \circ T(m), \\
\omega(x, y)=[x, y]_{N} & =\left[N_{1} x, y\right]+\left[x, N_{1} y\right]-N_{1}[x, y] ,\\
\nu(x, m)=[x, m]_{N} & =\left[N_{1} x, m\right]+\left[x, N_{0} m\right]-N_{0}[x, m].
\end{aligned}
$$

\begin{definition}
 A pair of linear maps $N=\left(N_{0}, N_{1}\right)$ is called a Nijenhuis operator if for all $x, y \in \fg$ and $m \in M$, the following conditions are satisfied:

(i) $\operatorname{Im}\left(T \circ N_{1}-N_{1} \circ T\right) \in \operatorname{Ker} N_{1}$;

(ii) $N_{1}[x, y]_{N}=\left[N_{1} x, N_{1} y\right]$;

(iii) $N_{0}[x, m]_{N}=\left[N_{1} x, N_{0} m\right]$.
\end{definition}

\begin{definition}
A deformation is said to be trivial if there exists a pair of linear maps $N_{0}: M \rightarrow M , N_{1}: \fg \rightarrow \fg$, such that $\left(T_{0}, T_{1}\right)$ is a morphism from $\left(M_{\lambda}, \fg_{\lambda}, T_{\lambda},[\cdot, \cdot]_{\lambda}\right)$ to $(M, \fg, T,[\cdot, \cdot])$, where $T_{0}=\mathrm{id}+\lambda N_{0}, T_{1}=\mathrm{id}+\lambda N_{1}$.
\end{definition}

Note that $\left(T_{0}, T_{1}\right)$ is a morphism means that
\begin{align}
\label{(26)}
T\circ T_{0}(m) & =T_{1} \circ T_{\lambda}(m), \\
\label{(27)}
T_{1}[x, y]_{\lambda} & =\left[T_{1} x, T_{1} y\right], \\
\label{(28)}
T_{0}[x, m]_{\lambda} & =\left[T_{1} x, T_{0} m\right].
\end{align}

Now we consider what conditions that $N=\left(N_{0}, N_{1}\right)$ should satisfy. From \eqref{(26)}, we have
$$
T(m)+\lambda T \circ N_{0}(m)=\left(\mathrm{id}+\lambda N_{1}\right)(T(m)+\lambda \theta(m))=T(m)+\lambda N_{1}(T(m))+\lambda \theta(m)+\lambda^{2} N_{1} \theta(m) .
$$
Thus, we have
\begin{align*}
&\theta(m)=\left(T \circ N_{0}-N_{1} \circ T\right)(m), \\
&N_{1} \theta(m)=0.
\end{align*}
It follows that $(N_0,N_1)$ must satisfy the following condition:
\begin{align}
N_{1}\left(T \circ N_{0}-N_{1} \circ T\right)=0.
\end{align}
For \eqref{(27)}, the left hand side is equal to
$$
[x, y]+\lambda N_{1}([x, y])+\lambda \omega(x, y)+\lambda^{2} N_{1} \omega(x, y),
$$
and the right hand side is equal to
$$
[x, y]+\lambda\left[N_{1} x, y\right]+\lambda\left[x, N_{1} y\right]+\lambda^{2}\left[N_{1} x, N_{1} y\right].
$$
Thus, \eqref{(27)} is equivalent to
$$
\begin{aligned}
& \omega(x, y)=\left[N_{1} x, y\right]+\left[x, N_{1} y\right]-N_{1}[x, y], \\
& N_{1} \omega(x, y)=\left[N_{1} x, N_{1} y\right].
\end{aligned}
$$
It follows that $N_1$ must satisfy the following condition:
\begin{align}
\left[N_{1} x, N_{1} y\right]-N_{1}\left[N_{1} x, y\right]-N_{1}\left[x, N_{1} y\right]+N_{1}^{2}[x, y]=0.
\end{align}
For \eqref{(28)}, the left hand side is equal to
$$
[x, m]+\lambda \nu(x, m)+\lambda N_{0}([x, m])+\lambda^{2} N_{0} \nu(x, m),
$$
and the right hand side is equal to
$$
[x, m]+\lambda\left[N_{1}(x), m\right]+\lambda\left[x, N_{0}(m)\right]+\lambda^{2}\left[N_{1}(x), N_{0}(m)\right].
$$
Thus, \eqref{(28)} is equivalent to
$$
\begin{gathered}
\nu(x, m)=\left[N_{1} x, m\right]+\left[x, N_{0} m\right]-N_{0}[x, m], \\
N_{0} \nu(x, m)=\left[N_{1} x, N_{0} m\right]+N_{2}(x, \theta(m)) .
\end{gathered}
$$
It follows that $N$ must satisfy the following condition:
\begin{align}
\left[N_{1} x, N_{0} m\right]-N_{0}\left[N_{1} x, m\right]-N_{0}\left[x, N_{0} m\right]+N_{0}^{2}[x, m]=0 .
\end{align}

A Nijenhuis operator $\left(N_{0}, N_{1}\right)$ could give a trivial deformation by setting
\begin{align}
(\theta, \omega, \nu)=D\left(N_{0}, N_{1}\right) .
\end{align}

\begin{theorem}
 Let $N=\left(N_{0}, N_{1}\right)$ be a Nijenhuis operator. Then a deformation can be obtained by putting
$$
\left\{\begin{aligned}
\theta(m) & =\left(T \circ N_{0}-N_{1} \circ T\right)(m), \\
\omega(x, y) & =\left[N_{1} x, y\right]+\left[x, N_{1} y\right]-N_{1}[x, y] \\
\nu(x, m) & =\left[N_{1} x, m\right]+\left[x, N_{0} m\right]-N_{1}[x, m].
\end{aligned}\right.
$$
Furthermore, this deformation is trivial.
\end{theorem}

\begin{proof}
 Since $(\theta, \omega, \nu)=D\left(N_{0}, N_{1}\right)$, it is obvious that $(\theta, \omega, \nu)$ is a 2-cocycle. It is easy to check that $(\theta, \omega, \nu)$ defines an embedding tensor on Malcev algebra. Thus by Theorem $3.1,(\theta, \omega, \nu)$ generates a deformation.
 \end{proof}

 Now we consider the general formal deformations of any order. Let $(M, \fg, T)$ be an embedding tensor on Malcev algebra and $\theta_{i}: M \rightarrow \fg, \omega_{i}: \wedge^{2} \fg \rightarrow \fg, \nu_{i}: \fg \otimes M\rightarrow M, i \geqslant 0$ be linear maps where $\theta_{0}=$ $T, \omega_{0}(x, y)=[x, y]_{\fg}, \nu_{0}(x, m)=\rho(x)(m)$. Consider a $\lambda$-parametrized family of linear operations:
\begin{align}
T_{\lambda}(m) & \triangleq T(m)+\lambda \theta_{1}(m)+\lambda^{2} \theta_{2}(m)+\cdots, \\
\omega_{\lambda}(x, y) & \triangleq[x, y]+\lambda \omega_{1}(x, y)+\lambda^{2} \omega_{2}(x, y)+\cdots, \\
\nu_{\lambda}(x,m) & \triangleq \rho(x)(m)+\lambda \nu_{1}(x, m)+\lambda^{2} \nu_{2}(x, m)+\cdots.
\end{align}
In order that $\left(M_{\lambda}, \fg_{\lambda}, f_{\lambda}\right)$ forms an embedding tensor on Malcev algebra, we must have
\begin{align}
&\omega_{\lambda}(\omega_{\lambda}(x, z), \omega_{\lambda}(y, t))-\omega_{\lambda}(\omega_{\lambda}(\omega_{\lambda}(x, y), z), t)-\omega_{\lambda}(\omega_{\lambda}(\omega_{\lambda}(y, z), t), x) \nonumber \\
&-\omega_{\lambda}(\omega_{\lambda}(\omega_{\lambda}(z, t), x), y)-\omega_{\lambda}(\omega_{\lambda}(\omega_{\lambda}(t, x), y), z)=0,\\
\nonumber \\
&\nu_{\lambda}(\omega_{\lambda}(\omega_{\lambda}(x, y), z),m)-\nu_{\lambda}(x, \nu_{\lambda}(y, \nu_{\lambda}(z, m)))+\nu_{\lambda}(z, \nu_{\lambda}(x, \nu_{\lambda}(y, m))) \nonumber \\
&-\nu_{\lambda}(y, \nu_{\lambda}(\omega_{\lambda}(z, x), m))+\nu_{\lambda}(\omega_{\lambda}(y, z), \nu_{\lambda}(x, m))=0,\\
\nonumber \\
&\omega_{\lambda}(T_{\lambda}(m),T_{\lambda}(n))-T_{\lambda} \big(\nu_{\lambda}(T_{\lambda}(m), n)\big)=0.
\end{align}
which implies that
\begin{align}
\sum_{i+j+k=l}\big(&\omega_{i}(\omega_{j}(x,z),\omega_{k}(y,t))-\omega_{i}(\omega_{j}(\omega_{k}(x, y), z), t)-\omega_{i}(\omega_{j}(\omega_{k}(y, z), t), x) \nonumber \\
\label{41}&-\omega_{i}(\omega_{j}(\omega_{k}(z, t), x), y)-\omega_{i}(\omega_{j}(\omega_{k}(t, x), y), z)=0,\\
\nonumber \\
\sum_{i+j+k=l} \big(&\nu_{i}(\omega_{j}(\omega_{k}(x, y), z)),m)-\nu_{i}(x, \nu_{j}(y, \nu_{k}(z, m)))+\nu_{i}(z, \nu_{j}(x, \nu_{k}(y, m))) \nonumber \\
\label{42}&-\nu_{i}(y, \nu_{j}(\omega_{k}(z, x), m))+\nu_{i}(\omega_{j}(y, z), \nu_{k}(x,m))\big)=0,\\
\nonumber \\
\label{43}\sum_{i+j+k=l} \big(&\omega_{i}(T_{j}(m),T_{k}(n))-T_{i} \big(\nu_{j}(T_{k}(m), n)\big)\big)=0.
\end{align}
For $l=0$, conditions \eqref{41}, \eqref{42} and \eqref{43} are equivalent to that $\left(\omega_{0}, \nu_{0}, \theta_{0}\right)$ is an embedding tensor on Malcev algebra.
For $l=1$, these conditions are equivalent to  that  $\left(\omega_{1}, \nu_{1}, \theta_{1}\right)$ is a 2-cocycle.

\begin{definition}
The 2-cochain $\left(\omega_{1}, \nu_{1}, \theta_{1}\right)$ is called the infinitesimal of $\left(\omega_{\lambda}, \nu_{\lambda}, T_{\lambda}\right)$. More generally, if $\left(\omega_{i}, \nu_{i}, T_{i}\right)=0$ for $1 \leqslant i \leqslant(n-1)$, and $\left(\omega_{n}, \nu_{n}, T_{n}\right)$ is a non-zero cochain in $C^{2}((M, \fg, T),(M, \fg, T))$, then $\left(\omega_{n}, \nu_{n}, T_{n}\right)$ is called the $n$-infinitesimal of the deformation $\left(\omega_{\lambda}, \nu_{\lambda}, T_{\lambda}\right)$.
\end{definition}

Let $\left(\omega_{\lambda}, \nu_{\lambda}, T_{\lambda}\right)$ and $\left(\omega_{\lambda}^{\prime}, \nu_{\lambda}^{\prime}, T_{\lambda}^{\prime}\right)$ be two deformation. We say that they are equivalent if there exists a formal isomorphism $\left(\Phi_{\lambda}, \Psi_{\lambda}\right):\left(M_{\lambda}^{\prime}, \fg_{\lambda}^{\prime}, T_{\lambda}^{\prime}\right) \rightarrow\left(M_{\lambda}, \fg_{\lambda}, T_{\lambda}\right)$ such that $\omega_{\lambda}^{\prime}(x, y)=$ $\Psi_{\lambda}^{-1} \omega_{\lambda}\left(\Psi_{\lambda}(x), \Psi_{\lambda}(y)\right)$

A deformation $\left(\omega_{\lambda}, \nu_{\lambda}, T_{\lambda}\right)$ is said to be the trivial deformation if it is equivalent to $\left(\omega_{0}, \nu_{0}, \theta_{0}\right)$.

\begin{theorem}\label{thm4.6}
Let $\left(\omega_{\lambda}, \nu_{\lambda}, T_{\lambda}\right)$ and $\left(\omega_{\lambda}^{\prime}, \nu_{\lambda}^{\prime}, T_{\lambda}^{\prime}\right)$ be equivalent deformations of $(M, \fg, T)$, then the first-order terms of them belong to the same cohomology class in the second cohomology group $H^{2}((M, \fg, T),(M, \fg, T))$.
\end{theorem}

\begin{proof}
Let $\left(\Phi_{\lambda}, \Psi_{\lambda}\right):\left(M_{\lambda}, \fg_{\lambda}, T_{\lambda}\right) \rightarrow\left(M_{\lambda}^{\prime}, \fg_{\lambda}^{\prime}, T_{\lambda}^{\prime}\right)$ be an equivalence where $\Phi_{\lambda}=\operatorname{id}_{M}+\lambda \phi_{1}+$ $\lambda^{2} \phi_{2}+\cdots$ and $\Psi_{\lambda}=\operatorname{id}_{M}+\lambda \psi_{1}+\lambda^{2} \psi_{2}+\cdots$. Then we have $\Psi_{\lambda} \omega_{\lambda}^{\prime}(x, y)=\omega_{\lambda}\left(\Psi_{\lambda}(x), \Psi_{\lambda}(y)\right)$, $\Psi_{\lambda} \nu_{\lambda}^{\prime}(x, v)=\nu_{\lambda}\left(\Phi_{\lambda}(x), \Psi_{\lambda}(v)\right)$. Then by expanding the above equality, we have $\left(\theta_{1}, \omega_{1}, \nu_{1}\right)-$ $\left(\theta_{1}^{\prime}, \omega_{1}^{\prime}, \nu_{1}^{\prime}\right)=D\left(\phi_{1}, \psi_{1}\right)$. Thus $\left(\theta_{1}, \omega_{1}, \nu_{1}\right)$ and $\left(\theta_{1}^{\prime}, \omega_{1}^{\prime}, \nu_{1}^{\prime}\right)$ are belong to the same cohomology class in the second cohomology group. The proof is finished.
\end{proof}

An embedding tensor $(M, \fg, T)$ is called rigid if every deformation $\left(\omega_{\lambda}, \nu_{\lambda}, T_{\lambda}\right)$ is equivalent to the trivial deformation.

\begin{theorem}
If $H^{2}((M, \fg, T),(M, \fg, T))=0$, then $(M, \fg, T)$ is rigid.
\end{theorem}

\begin{proof}
Let $\left(\omega_{\lambda},\nu_{\lambda}, T_{\lambda}\right)$ be a deformation of $(M, \fg, T)$. It follows from above Theorem \ref{thm4.6} that $D\left(\omega_{\lambda}, \nu_{\lambda}, T_{\lambda}\right)=0$, that is $\left(\omega_{\lambda}, \nu_{\lambda}, T_{\lambda}\right) \in Z^{2}((M, \fg, T),(M, \fg, T))$. Now we assume $H^{2}((M, \fg, T),(M, \fg, T))=0$, we can find $\left(N_{0}, N_{1}\right)$ such that $\left(\omega_{\lambda}, \nu_{\lambda}, T_{\lambda}\right)=D\left(N_{0}, N_{1}\right)$. Thus $\left(\omega_{\lambda}, \nu_{\lambda}, T_{\lambda}\right)$ is equivalent to the trivial deformation. This proof is completed.
\end{proof}

\vskip7pt
\footnotesize{
\noindent Tao Zhang\\
College of Mathematics and Information Science,\\
Henan Normal University, Xinxiang 453007, P. R. China;\\
 E-mail address: \texttt{{zhangtao@htu.edu.cn}}

\vskip7pt
\footnotesize{
\noindent Wei Zhong\\
College of Mathematics and Information Science,\\
Henan Normal University, Xinxiang 453007, P. R. China};\\
 E-mail address: \texttt{{zhongweiHTU@yeah.net }}


\begin{thebibliography}{Bre}

%
%
%


\bibitem{Bre}
M. R. Bremner, L. A. Peresi and J. Sanchez-Ortega, Malcev dialgebras,  \emph{Linear Multilinear Algebra} 60(10)(2012), 1125-1141.



\bibitem{bon2} R. Bonezzi and O. Hohm, Duality hierarchies and differential graded Lie algebras,  {\em Commun. Math. Phys.} 382 (2021), 277--315.

\bibitem{CPal}
M. Cederwall and J. Palmkvist,
Tensor hierarchy algebra extensions of over-extended Kac-Moody algebras,  {\em Comm. Math. Phys.} 389 (2022), 571--620.


\bibitem{Das1}
A. Das, Embedding tensors on Hom-Lie algebras, 2023, arXiv:2304.04178v1.




\bibitem{E1}
A. Elduque, On semisimple Malcev Algebras, {\em Proc. Amer. Math. Soc.} 107 (1989), 73--82.




\bibitem{Fi}
V.T. Filippov, Mal'tsev algebras,  \emph{Algebra and Logic} 16(1) (1977), 70--74.


\bibitem{Ge}
M. Gerstenhaber, On the deformation of rings and algebras. \emph{Ann. Math. (2) } {79} (1964), 59-103.






\bibitem{Mab1}
F. Harrathi, S. Mabrouk, O. Ncib, S. Silvestrov,
Kupershmidt operators on Hom-Malcev algebras and their deformation,
 \emph{Int. J. Geom. Meth. Mod. Phys.} 20(03)(2023), 2350046

\bibitem{HHSZ}
M. Hu, S. Hou, L. Song and Y. Zhou,
Deformations and cohomologies of embedding tensors on 3-Lie algebras, 2023, arXiv:2302.08725. to appear in Comm. Algebra.


\bibitem{kotov-strobl} A. Kotov and T. Strobl, The embedding tensor, Leibniz-Loday algebras and their higher gauge theories, {\em Comm. Math. Phys.} 376 (2020), 235--258.


\bibitem{Kuz}
E.N. Kuz'min, Mal'tsev algebras and their representations,  \emph{Algebra and Logic} { 7} (4) (1968), 233--244.


\bibitem{lavau} S. Lavau, Tensor hierarchies and Leibniz algebras, {\em J. Geom. Phys.} 144 (2019), 147--189.

\bibitem{lavau-p} S. Lavau and J. Palmkvist, Infinity-enhancing Leibniz algebras, {\em Lett. Math. Phys.} 110 (2020), 3121--3152.

\bibitem{lavau-stas} S. Lavau and J. Stasheff, From Lie algebra crossed modules to tensor hierarchies, {\em J. Pure Appl. Algebra} 227(2023),107311


\bibitem{Mab}
S. Mabrouk, Deformation of Kupershmidt operators and  Kupershmidt-Nijenhuis structures of a Malcev algebra,
 \emph{Hacettepe J. Math.  Stat.}, 51(1)(2022), 199--217.


\bibitem{Mal}
 A. Malcev, Analytic loops, \emph{Mat. Sb.} { 78}(1955), 569--578.



\bibitem{nicolai} H. Nicolai and H. Samtleben, Maximal gauged supergravity in three dimensions,
{\em Phys. Rev. Lett.} 86(9)(2001), 1686--1689.

\bibitem{NR} A. Nijenhuis  and R. Richardson,
Cohomology and deformations in graded Lie algebras. {\em Bull. Amer. Math. Soc.} { 72} (1966), 1--29.

%






\bibitem{Pal}
J. Palmkvisk, The tensor hierarchy algebra, {\em J. Math. Phys.} 55 (2014), 011701.



\bibitem{Sag}
A. A. Sagle, Malcev algebras, \emph{Trans. Amer. Math. Soc.} { 101}(1961), 426--458.

\bibitem{sheng-embd} Y. Sheng, R. Tang and C. Zhu, The controlling $L_\infty$-algebra, cohomology and homotopy of embedding tensors and Lie-Leibniz triples, {\em Comm. Math. Phys.} 386 (2021), 269--304.




\bibitem{TY}  R. Tang, Y. Sheng, Nonabelian embedding tensors,   {\em Lett. Math. Phys.} 113 (2023), 14.


\bibitem{wit}
B. de Wit, H. Samtleben and M. Trigiante,
On Lagrangians and gaugings of maximal supergravities,
{\em Nuclear Phys. B} 655 (2003), no. 1-2, 93--126.

\bibitem{wit2}
B. de Wit, H. Samtleben and M. Trigiante,
The maximal $D = 5$ supergravities, {\em Nucl. Phys. B} 716 (2005), 215--247.

\bibitem{wit3}
B. de Wit and H. Samtleben, Gauged maximal supergravities and hierarchies of nonabelian vector-tensor systems, {\em Fortschr. Phys.} 53 (2005), 442--449.

\bibitem{wit4}
B. de Wit, H. Nicolai and H. Samtleben, Gauged supergravities, tensor hierarchies, and $M$-theory, {\em J. High Energy Phys.} 02 (2008), 044, 33 pp.


\bibitem{Yam1}
K. Yamaguti, Note on Malcev algebras, \emph{Kumamoto J. Sci. Ser. A}  5 (1962), 203--207.

\bibitem{Yam2}
K. Yamaguti, On the theory of Malcev algebras, \emph{Kumamoto J. Sci. Ser. A}   6 (1963), 9--45.










%












%




\end{thebibliography}
\end{document}